\documentclass[12pt]{article}

\usepackage[authoryear]{natbib}
\usepackage{graphicx}
\usepackage{amssymb, amsmath, amsthm}
\usepackage{times}
\usepackage{bm}
\usepackage{color}

\newcommand{\Es}{E}

\newcommand{\e}{{\rm e}}
\newcommand{\dd}{{\rm d}}

\newtheorem{lemma}{Lemma}
\newtheorem{Teo}{Theorem}
\newtheorem{col}{Corollary}
\newtheorem{ex}{Example}

\title{Gradient statistic: higher-order asymptotics and Bartlett-type correction}
\author{Tiago M.~Vargas,\quad Silvia L.P.~Ferrari,\quad Artur J.~Lemonte\\
{\small {\em Departamento de Estat\'istica, Universidade de S\~ao Paulo, S\~ao Paulo/SP, Brazil}}}
\date{}

\begin{document}
\maketitle

\begin{abstract}

We obtain an asymptotic expansion for the null distribution function of the
gradient statistic for testing composite null hypotheses in the presence of
nuisance parameters. The expansion is derived using a Bayesian
route based on the shrinkage argument described in \cite{GhoshMukerjee1991}. 
Using this expansion, we propose a Bartlett-type corrected gradient statistic with
chi-square distribution up to an error of order $o(n^{-1})$ under the null hypothesis.
Further, we also use the expansion to modify the percentage points of the large sample reference
chi-square distribution. A small Monte Carlo experiment and
various examples are presented and discussed. \\

\noindent {\it Key-words:}  Asymptotic expansion; Bartlett-type correction; Bayesian route;
Gradient statistic; Shrinkage argument.
\end{abstract}

\section{Introduction}

The most common hypothesis tests for large samples are the likelihood ratio \citep{Wilks1938},
the Wald \citep{Wald1943}, and the Rao score \citep{Rao1948} tests.
These tests are widely used in areas such as
economics, biology, and engineering, among others, since exact tests are not always
available. An alternative test uses the gradient statistic recently
proposed by \cite{Terrell2002}. An advantage of the gradient statistic
over the Wald and the score statistics is that it does not involve
knowledge of the information matrix, neither expected nor observed.
Additionally, the gradient statistic is quite simple to be computed. This has been
emphasised by C.R.~Rao \citep{Rao2005}, who wrote:
`The suggestion by Terrell is attractive as it
is simple to compute. It would be of interest to investigate
the performance of the [gradient] statistic'.

Let $x_{1},\ldots,x_{n}$ be a random sample of size $n$
with joint probability density function $f(\cdot;\bm{\theta})$, which depends
on a $p$-dimensional vector of unknown parameters $\bm{\theta}=(\theta_{1},\ldots,\theta_{p})^{\top}$.
Let $\ell(\bm{\theta})=n^{-1}\sum_{i=1}^{n} \log f(x_{i};\bm{\theta})$
and $\bm{U}(\bm{\theta}) = \partial\ell(\bm{\theta})/\partial\bm{\theta}$
be the log-likelihood function and the score vector, respectively;
notice that, for convenience, both are divided by $n$.
We wish to test the null hypothesis $\mathcal{H}_{0}: \bm{\theta}_{1}=\bm{\theta}_{10}$
against the two-sided alternative hypothesis $\mathcal{H}_{a}:\bm{\theta}_{1}\neq \bm{\theta}_{10}$,
where $\bm{\theta}_{10}$ is a fixed $q$-dimensional vector,
$\bm{\theta}_{1}=(\theta_{1},\ldots,\theta_{q})^{\top}$ and
$\bm{\theta}_{2}=(\theta_{q+1},\ldots,\theta_{p})^{\top}$.
The partition in $\bm{\theta}$ induces the corresponding partition in $\bm{U}(\bm{\theta})$:
$\bm{U}(\bm{\theta}) = (\bm{U}_{1}(\bm{\theta})^{\top},
\bm{U}_{2}(\bm{\theta})^{\top})^{\top}$. Let
$\widehat{\bm{\theta}}=(\widehat{\bm{\theta}}_1,\widehat{\bm{\theta}}_2)^{\top}$ and
$\widetilde{\bm{\theta}}=(\bm{\theta}_{10},\widetilde{\bm{\theta}}_2)^{\top}$ be the unrestricted and the
restricted (under $\mathcal{H}_{0}$) maximum likelihood estimators of
$\bm{\theta}=(\bm{\theta}_{1}^\top,\bm{\theta}_{2}^\top)^\top$, respectively.
The gradient statistic for testing $\mathcal{H}_{0}$ is defined as
\begin{equation}\label{est_grad}
S=n\bm{U}(\widetilde{\bm{\theta}})^{\top}(\widehat{\bm{\theta}}-\widetilde{\bm{\theta}}),
\end{equation}
and can also be written as
$S = n\bm{U}_{1}(\widetilde{\bm{\theta}})^{\top}(\widehat{\bm{\theta}}_{1} - \bm{\theta}_{10})$,
since $\bm{U}_{2}(\widetilde{\bm{\theta}}) = \bm{0}$.
Like the likelihood ratio, the Wald, and the score statistics, the gradient statistic has
an asymptotic $\chi^{2}_{q}$ distribution under the null hypothesis, 
$q$ being the number of restrictions imposed by $\mathcal{H}_{0}$.

Equation \eqref{est_grad} is the inner
product of the score vector evaluated at $\mathcal{H}_{0}$
and the difference between the unrestricted and the restricted maximum likelihood estimators
of $\bm{\theta}$. Although the gradient statistic was derived by \cite{Terrell2002} from the score
and the Wald statistics, it is of a different nature. The score statistic measures the squared
length of the score vector evaluated at $\mathcal{H}_{0}$ using the metric given by the inverse of
the Fisher information matrix, whereas the Wald statistic gives the squared distance between the
unrestricted and the restricted maximum likelihood estimators
of  $\bm{\theta}$ using the metric given by the Fisher
information matrix. Moreover, both are quadratic forms. The gradient statistic, on the other hand,
is not a quadratic form and measures the distance between the unrestricted and the restricted
maximum likelihood estimators of  $\bm{\theta}$ from a different perspective.
It measures the orthogonal projection of the score vector at $\mathcal{H}_{0}$
on the vector $\widehat{\bm{\theta}}-\widetilde{\bm{\theta}}$.

Recently, the gradient test has been the subject of some research papers.
In particular, \cite{LemonteFerrari2012a} obtained the local power of the gradient
test under Pitman alternatives (a sequence of alternative hypotheses converging to the null
hypothesis at the rate of $n^{-1/2}$). The authors compared
the local power of the gradient test with those
of the likelihood ratio, the Wald, and the score tests. They showed that none of
the tests is uniformly  more powerful than the others, and
therefore, the gradient test is not only
very simple to be calculated but  it is also competitive with
the others in terms of local power.
Other recent works in which the gradient test
is investigated are \cite{Lemonte2011,Lemonte2012a} and
\cite{LemonteFerrari2011,LemonteFerrari2012b,LemonteFerrari2012c}.

The main result in \cite{LemonteFerrari2012a} regarding the local power of the gradient test up
to an error of order $o(n^{-1/2})$ represents the first step in the study of higher
order asymptotic properties of the gradient test. In the present paper, we wish to go further by
focusing on deriving the second-order approximation to the null distribution of the gradient statistic.
In other words, our aim is to obtain an asymptotic expansion for the cumulative distribution function of
the gradient statistic under the null hypothesis up to an error of order $o(n^{-1})$.

The usual route for deriving expansions for the distribution of asymptotic chi-square test statistics
involves multivariate Edgeworth series expansions. Although such a route has been followed by many authors,
it is extremely lengthy and tedious \citep[see, for example,][]{Hayakawa1977,Harris1985}.
Here, on the other hand, in order to derive an asymptotic expansion for the null distribution
of the gradient statistic up to order $n^{-1}$, we
follow a Bayesian route based on a shrinkage argument
originally suggested by  \cite{GhoshMukerjee1991} and described later
in \cite{MukerjeeReid2000}. Although it uses a Bayesian approach,
this technique can be used to solve frequentist problems,
such as the derivation of Bartlett corrections and tail probabilities \citep{DattaMukerjee2003}.

Additionally, we obtain a Bartlett-type
correction factor for the gradient statistic from the results in
\cite{CordeiroFerrari1991}. Under the null hypothesis, the corrected statistic is
distributed  as chi-square up to an error of order $o(n^{-1})$, while the uncorrected gradient
statistic has a chi-square distribution
up to an error of order $o(n^{-1/2})$; that is, the Bartlett-type correction factor
makes the approximation error be reduced from $o(n^{-1/2})$ to $o(n^{-1})$.
For a detailed survey on Bartlett and Bartlett-type corrections,
the reader is referred to \cite{CordeiroCribari1996}.

The paper unfolds as follows. In Section \ref{main_results},
we present our main results, namely an asymptotic expansion
for the cumulative distribution function of the
gradient statistic and its Bartlett-type correction.
In Sections \ref{onefamily} and \ref{twofamily}, we
particularise our general results to one-parameter families and to families
with two orthogonal parameters, respectively.
A small Monte Carlo study is also presented in Section \ref{twofamily}.
Section \ref{conclusion} closes the paper with a brief discussion.
Technical details are collected in two appendices.

\section{The main result}\label{main_results}

First, let us introduce some notation. Let $D_{j}=\partial/\partial\theta_{j}$
($j=1,\ldots,p$) be the differential operator.
We define $U_j=D_j\ell(\bm{\theta})$,
$U_{jr}=D_jD_r\ell(\bm{\theta})$, $U_{jrs}=D_jD_rD_s\ell(\bm{\theta})$,
and so on. We make the same assumptions, such as the regularity of the first four
derivatives of $\ell(\bm{\theta})$ with respect to $\bm{\theta}$
and the existence and uniqueness of the maximum likelihood estimator of $\bm{\theta}$,
as those fully outlined by \cite{Hayakawa1977}.
Let $\kappa_{j,r}=\Es(U_{j}U_{r})$, $\kappa_{jr}=\Es(U_{jr})$, $\kappa_{jrs}=\Es(U_{jrs})$,
$\kappa_{jrsu}=\Es(U_{jrsu})$, $\kappa_{j,rs}=\Es(U_{j}U_{rs})$,
$\kappa_{jrs,u}=\Es(U_{jrs}U_{u})$, $\kappa_{ju,rs}=\Es(U_{ju}U_{rs})-\kappa_{ju}\kappa_{rs}$,
$\kappa_{j,u,rs}=\Es(U_j U_u U_{rs})+\kappa_{ju}\kappa_{rs}$, etc., denote
the cumulants of log-likelihood derivatives.
The cumulants are not functionally independent, for instance,
$\kappa_{j,r} = -\kappa_{jr}$, $\kappa_{jr,s} + \kappa_{jrs} = \kappa_{jr}^{(s)}$,
$\kappa_{j,rsu} + \kappa_{jrsu} =\kappa_{rsu}^{(j)}$,
$\kappa_{j,r,su} = \kappa_{jrsu} - \kappa_{jsu}^{(r)} + \kappa_{su}^{(jr)} - \kappa_{jr,su}$,
where $\kappa_{jr}^{(s)}=D_s\kappa_{jr}$ and $\kappa_{su}^{(jr)}=D_jD_r\kappa_{su}$, etc.
Relations among them were first obtained by
\cite{Bartlett1953a, Bartlett1953b}.
Further, let $\bm{K}$ be the Fisher information matrix
\[
\bm{K}=((\kappa_{j,r}))=-((\kappa_{jr})) =
\begin{bmatrix}
\bm{K}_{11} & \bm{K}_{12} \\
\bm{K}_{21} & \bm{K}_{22}
\end{bmatrix},
\]
with $\bm{K}^{-1}=((\kappa^{j,r}))$ denoting its inverse.  Finally, define the
matrices
\[
\bm{A} = ((a^{jr})) =
\begin{bmatrix}
\bm{0} & \bm{0} \\
\bm{0} & \bm{K}_{22}^{-1}
\end{bmatrix},
\qquad
\bm{M} = ((m^{jr})) = \bm{K}^{-1} - \bm{A}.
\]
In what follows, we use the Einstein summation convention, where
$\sum^{\prime}$ denotes summation over all components
of $\bm{\theta}$; that is, the indices $j$, $r$, $s$, $k$, $l$ and $u$ range
over $1$ to $p$. We now establish the following theorem.
\begin{Teo}\label{theorem1}
The asymptotic expansion for the null distribution of the gradient statistic for
testing $\mathcal{H}_{0}: \bm{\theta}_{1}=\bm{\theta}_{10}$ against
$\mathcal{H}_{a}:\bm{\theta}_{1}\neq \bm{\theta}_{10}$ is
\begin{equation}\label{cdf_grad}
\Pr(S\leq x)=G_q(x)+ \frac{1}{24n}\sum_{i=0}^{3}R_{i}G_{q+2i}(x)+o(n^{-1}),
\end{equation}
where $G_z(x)$ is the cumulative distribution function of a chi-square random variable with $z$ degrees of freedom,
$R_{1}=3A_3-2A_2+A_1$, $R_{2}=A_2 -3A_3$, $R_{3}=A_3$, $R_0=-(R_1+R_2+R_3)$,
\begin{align*}
A_{1}&= 3\sum\nolimits^{\prime}\kappa_{jrs}\kappa_{klu}a^{lu}\bigl\{3m^{jk}a^{rs}+m^{jr}\bigl(\kappa^{s,k}+2a^{sk}\bigr)\bigr\}\\
&\quad+ 6\sum\nolimits^{\prime}\kappa_{jrs,u}m^{jr}a^{su}
-6\sum\nolimits^{\prime}\bigl(\kappa_{jrsu}+\kappa_{jrs,u}\bigr)\bigl(m^{jr}\kappa^{s,u}+2m^{ju}a^{rs}\bigr)\\
&\quad+6\sum\nolimits^{\prime}\bigl(\kappa_{klu}+\kappa_{kl,u}\bigr)
\Bigl[2\bigl(\kappa_{jrs}+\kappa_{jr,s}\bigr)\bigl(\kappa^{s,j}\kappa^{r,k}\kappa^{l,u}-a^{sj}a^{rk}a^{lu}\\
&\qquad +\kappa^{s,k}\kappa^{l,j}\kappa^{r,u}-a^{sk}a^{lj}a^{ru}\bigr)
-\kappa_{jrs}\bigl\{\bigl(\kappa^{s,u}+a^{su}\bigr)\bigl(\kappa^{j,k}\kappa^{l,r}-a^{jk}a^{lr}\bigr)\\
&\qquad+m^{jr}\bigl(a^{sk}a^{lu}+\kappa^{s,k}\kappa^{l,u}\bigr)
+2a^{rs}\bigl(\kappa^{j,k}\kappa^{l,u}-a^{jk}a^{lu}\bigr)+2a^{rk}a^{ls}m^{ju}\bigr\}\Bigr]\\
&\quad +12\sum\nolimits^{\prime}\bigl(\kappa_{jrsu}+\kappa_{j,rsu}+\kappa_{jsu,r}
+\kappa_{ju,rs}+\kappa_{j,u,rs}\bigl)\bigr(\kappa^{j,s}\kappa^{u,r}-a^{js}a^{ur}\bigr),
\end{align*}
\begin{align*}
A_{2}&=-3\sum\nolimits^{\prime}\kappa_{jrs}\biggl[\kappa_{klu}\biggl\{m^{jr}\biggl(m^{sk}a^{lu}
+\frac{3}{4}m^{sk}m^{lu}+3m^{kl}a^{su}\biggr)+\frac{1}{2}m^{jk}m^{rl}m^{su}\biggr\}\\
&\qquad-2\bigl(\kappa_{klu}+\kappa_{kl,u}\bigr)\bigl\{m^{su}\bigl(\kappa^{j,k}\kappa^{l,r}
- a^{jk}a^{lr}\bigr)+m^{jr}\bigl(\kappa^{s,k}\kappa^{l,u}-a^{sk}a^{lu}\bigr)\bigr\}\biggr]\\
&\quad+3\sum\nolimits^{\prime}\bigl(\kappa_{jrsu}+2\kappa_{jrs,u}\bigr)m^{jr}m^{su},
\end{align*}
\[
A_{3}=\frac{1}{12}\sum\nolimits^{\prime}\kappa_{jrs}\kappa_{klu}\bigl(9m^{jr}m^{sk}m^{lu}+6m^{jk}m^{rl}m^{su}\bigr).
\]
\end{Teo}
\begin{proof}
The proof is presented in Appendix 1.
\end{proof}

Basically, in order to prove Theorem \ref{theorem1}, we follow a Bayesian
route based on a shrinkage argument. This argument
is described in Appendix 2.

If the null hypothesis is simple, we have $q=p$, $\bm{A} = \bm{0}$
and $\bm{M}=\bm{K}^{-1}$. Therefore, an immediate consequence
of Theorem \ref{theorem1} is the following corollary.
\begin{col}\label{corollary1}
The asymptotic expansion for the null distribution of the gradient statistic for
testing $\mathcal{H}_{0}: \bm{\theta}=\bm{\theta}_{0}$ against
$\mathcal{H}_{a}:\bm{\theta}\neq \bm{\theta}_{0}$ is  given by \eqref{cdf_grad}
with $q=p$, $R_{1}=3A_3-2A_2+A_1$, $R_{2}=A_2 -3A_3$, $R_{3}=A_3$, $R_0=-(R_1+R_2+R_3)$
and the $A$'s are $A_{3}=\sum\nolimits^{\prime}\kappa_{jrs}\kappa_{klu}\bigl(9\kappa^{j,r}\kappa^{s,k}\kappa^{l,u}
+6\kappa^{j,k}\kappa^{r,l}\kappa^{s,u}\bigr)/12$,
\begin{align*}
A_{1}&= -6\sum\nolimits^{\prime}\bigl(\kappa_{jrsu}+\kappa_{jrs,u}\bigr)\kappa^{j,r}\kappa^{s,u}\\
&\quad+6\sum\nolimits^{\prime}\bigl(\kappa_{klu}+\kappa_{kl,u}\bigr)\bigl\{2\bigl(\kappa_{jrs}+\kappa_{jr,s}\bigr)
\bigl(\kappa^{s,j}\kappa^{r,k}\kappa^{l,u}+\kappa^{s,k}\kappa^{l,j}\kappa^{r,u}\bigr)\\
&\qquad-\kappa_{jrs}\bigl(\kappa^{s,u}\kappa^{j,k}\kappa^{l,r}
+\kappa^{j,r}\kappa^{s,k}\kappa^{l,u}\bigr)\bigr\}\\
&\quad+12\sum\nolimits^{\prime}\bigl(\kappa_{jrsu}+\kappa_{j,rsu}+\kappa_{jsu,r}
+\kappa_{ju,rs}+\kappa_{j,u,rs}\bigr)\kappa^{j,s}\kappa^{u,r},
\end{align*}
\begin{align*}
A_{2}&=-3\sum\nolimits^{\prime}\kappa_{jrs}\biggl\{\kappa_{klu}\biggl(\frac{3}{4}\kappa^{j,r}
\kappa^{s,k}\kappa^{l,u}+\frac{1}{2}\kappa^{j,k}\kappa^{r,l}\kappa^{s,u}\biggr) \\
&\qquad-2\bigl(\kappa_{klu}+\kappa_{kl,u}\bigr)\bigl(\kappa^{s,u}\kappa^{j,k}\kappa^{l,r}
+\kappa^{j,r}\kappa^{s,k}\kappa^{l,u}\bigr)\biggr\}\\
&\quad+3\sum\nolimits^{\prime}\bigl(\kappa_{jrsu}+2\kappa_{jrs,u}\bigr)\kappa^{j,r}\kappa^{s,u}.
\end{align*}
\end{col}

We are now able to present a Bartlett-type corrected gradient statistic.
A Bartlett-type correction is a multiplying factor, which
depends on the statistic itself, that results in
a modified statistic that follows a chi-square distribution with approximation error
of order less than $n^{-1}$. \cite{CordeiroFerrari1991}
obtained a general formula for a Bartlett-type correction
for a wide class of statistics that have a chi-square
distribution asymptotically. A special case is when the
cumulative distribution function of the statistic can be written
as \eqref{cdf_grad}, independently of the coefficients $R_1$, $R_2$, and $R_3$.
Hence, from Theorem \ref{theorem1} and the results in \cite{CordeiroFerrari1991},
we have the following corollary.
\begin{col}\label{col_corr_grad}
The modified statistic
\begin{equation}\label{corr_grad}
S^{*}=S\bigl\{1-\bigl(c+bS+aS^{2}\bigr)\bigr\},
\end{equation}
where
\[
a=\frac{A_{3}}{12nq(q+2)(q+4)},\qquad
b=\frac{A_2 -2A_3}{12nq(q+2)}, \qquad
c=\frac{A_1-A_2+A_3}{12nq},
\]
has a $\chi^{2}_{q}$ distribution up to an error of order $o(n^{-1})$ under
the null hypothesis.
\end{col}

The factor $\{1-(c+bS+aS^{2})\}$ in \eqref{corr_grad} can be regarded as a Bartlett-type
correction factor for the gradient statistic in such a way that the
null distribution of $S^{*}$ is better approximated
by the reference $\chi^2$ distribution than the distribution of the uncorrected
gradient statistic.

Instead of modifying the test statistic as in \eqref{corr_grad},
we may modify the reference $\chi^2$ distribution using the inverse
expansion formula in \cite{HillDavis1968}. To be specific, let $\gamma$ be the desired 
level of the test, and $x_{1-\gamma}$ be the $1-\gamma$ percentile of 
the $\chi^2$ limiting distribution of the test statistic.
From expansion \eqref{cdf_grad}, we have the following corollary.
\begin{col}
The asymptotic expansion for the $1-\gamma$ percentile of $S$ to order $n^{-1}$ takes the form
\begin{align}\label{HD}
\begin{split}
z_{1-\gamma} &= x_{1-\gamma} + \frac{1}{12n}\biggl[\frac{A_{3}x_{1-\gamma}}
{q(q+2)(q+4)}\bigl\{x_{1-\gamma}^2 + (q+4)x_{1-\gamma}+ (q+2)(q+4)\bigr\}\\
&\qquad\qquad\quad +\frac{x_{1-\gamma}(x_{1-\gamma} + q + 2)}{q(q+2)}(A_{2} - 3A_{3}) + \frac{x_{1-\gamma}}{q}
(3A_{3} - 2A_{2} + A_{1})\biggr],
\end{split}
\end{align}
where $\Pr(\chi_{q}^2\geq x_{1-\gamma})=\gamma$.
\end{col}

In general, equations \eqref{corr_grad} and \eqref{HD} depend on unknown parameters.
In this case, we can replace these unknown parameters by their maximum likelihood
estimates obtained under $\mathcal{H}_{0}$. It should be noticed
that the improved gradient test of the null hypothesis $\mathcal{H}_{0}$ may be
performed in three ways: (i) by referring the corrected statistic $S^*$ in \eqref{corr_grad}
to the $\chi_q^2$ distribution; (ii) by referring the gradient statistic $S$ to the
approximate cumulative distribution function \eqref{cdf_grad}; (iii) by comparing $S$ with the modified upper
percentile in \eqref{HD}. These three procedures are equivalent to order $n^{-1}$.

Finally, the three moments, up to order $n^{-1}$ under the null
hypothesis, of the gradient statistic are presented in the following corollary.
\begin{col}\label{col_momen_grad}
The first three moments, up to order $n^{-1}$ under the null hypothesis, of the gradient statistics are
\[
\mu_{1}'(S)=q+\frac{A_{1}}{12n}, \qquad  \mu_{2}(S)=2q+\frac{A_1+A_2}{3n},
\]
\[
\mu_{3}(S)= 8q+\frac{2(A_1 +2A_2+A_3)}{n}.
\]
\end{col}

In the next sections, we consider some applications of the general results derived in
this section in two special cases: a one-parameter model and a two-parameter model
under orthogonality of parameters.

\section{The one-parameter case}\label{onefamily}

We initially assume that the model is indexed by a scalar unknown parameter, say $\phi$.
The interest lies in testing the null hypothesis
$\mathcal{H}_{0}:\phi=\phi_{0}$ against $\mathcal{H}_{a}:\phi\neq\phi_{0}$,
where $\phi_{0}$ is a fixed value.
Let $\kappa_{\phi\phi}=\Es(\partial^{2} \ell(\phi)/\partial\phi^{2})$,
$\kappa_{\phi\phi\phi}=\Es(\partial^{3} \ell(\phi)/\partial\phi^{3})$,
$\kappa_{\phi\phi\phi\phi}=\Es(\partial^{4} \ell(\phi)/\partial\phi^{4})$,
$\kappa_{\phi\phi}^{(\phi)}=\partial \kappa_{\phi\phi}/\partial\phi$,
$\kappa_{\phi\phi\phi}^{(\phi)}=\partial \kappa_{\phi\phi\phi}/\partial\phi$, and
$\kappa_{\phi\phi}^{(\phi\phi)}=\partial^{2} \kappa_{\phi\phi}/\partial\phi^{2}$.
The gradient statistic for testing $\mathcal{H}_{0}$ is
$S=nU(\phi_{0})(\widehat\phi-\phi_{0})$, where $\widehat{\phi}$
is the maximum likelihood estimator of $\phi$. Here, $A_1$, $A_2$, and $A_3$
given in Corollary \ref{corollary1} reduce to
\begin{equation}\label{c3}
A_{1}=\frac{6\kappa_{\phi\phi}(2\kappa_{\phi\phi}^{(\phi\phi)}
-\kappa_{\phi\phi\phi}^{(\phi)})+12\kappa_{\phi\phi}^{(\phi)}(\kappa_{\phi\phi\phi}
-2\kappa_{\phi\phi}^{(\phi)})}{\kappa^{3}_{\phi\phi}},
\end{equation}
\begin{equation}\label{c2}
A_{2}=\frac{12\kappa_{\phi\phi}(2\kappa_{\phi\phi\phi}^{(\phi)}
-3\kappa_{\phi\phi\phi\phi})+3\kappa_{\phi\phi\phi}(5\kappa_{\phi\phi\phi}
-16\kappa_{\phi\phi}^{(\phi)})}{4\kappa^{3}_{\phi\phi}},
\end{equation}
\begin{equation}\label{c1}
A_{3}=-\frac{5\kappa_{\phi\phi\phi}^{2}}{4\kappa^{3}_{\phi\phi}}.
\end{equation}

We now present some examples.
\begin{ex}(Exponential distribution)\end{ex}
Let $x_{1},\ldots,x_{n}$ be a random sample of an exponential
distribution with density
\[
f(x;\phi)=\displaystyle{\frac{1}{\phi}}\e^{-x/\phi},\qquad x>0,\qquad \phi>0.
\]
Here, $\kappa_{\phi\phi}=-\phi^{-2}$, $\kappa_{\phi\phi\phi}=4\phi^{-3}$,
and $\kappa_{\phi\phi\phi\phi}=-18\phi^{-4}$.
The gradient statistic assumes the form $S=n(\bar{x}-\phi_{0})^{2} / \phi_{0}^2$,
where $\bar{x}=n^{-1}\sum_{i=1}^{n} x_i$, which equals the score statistic. It is easy to
see that $A_1=0$, $A_{2}=18$, and $A_{3}=20$. The first three moments (up to order $n^{-1}$) of $S$
are  $\mu_{1}'(S)=1$, $\mu_{2}(S)=2+6/n$, and $\mu_{3}(S)= 8+112/n$.
A partial verification of our results can be accomplished by comparing the exact moments of
$S$ with the approximate moments given above.  Since $n\bar{X}$ has a gamma distribution with parameters
$n$ and $1/(n\phi)$, it can be shown that the first three exact moments of $S$ are $1$, $2+6/n$,
and $8+112/n+120/n^2$, respectively. These moments differ from the approximate moments
obtained from Corollary \ref{col_momen_grad} only in terms of order less than $n^{-1}$. The
Bartlett-type corrected gradient statistic obtained from Corollary \ref{corr_grad}
is $S^{*}=S\{1-(3-11S+2S^{2})/(18n)\}$.

\begin{ex}(One-parameter exponential family)\end{ex}
Let $x_{1},\ldots,x_{n}$ be
a random sample of size $n$ in which each $x_i$ has a distribution in the one-parameter
exponential family with density
\[
f(x;\phi)=\frac{1}{\xi(\phi)}\exp{\{-\alpha(\phi)d(x)+v(x)\}},
\]
where $\alpha(\cdot)$, $v(\cdot)$, $d(\cdot)$, and $\xi(\cdot)$ are known functions.
Also, $\alpha(\cdot)$ and $\xi(\cdot)$ are assumed to have first three continuous derivatives,
with $\xi(\cdot) > 0$, $\alpha'(\phi)$, and $\beta'(\phi)$ being different from zero for all
$\phi$ in the parameter space, where
$\beta(\phi) = \xi'(\phi)/\{\xi(\phi)\alpha'(\phi)\}$.
Here, primes denote derivatives with respect to $\phi$. For instance,
$\beta'=\beta'(\phi)={\rm d}\beta(\phi)/{\rm d}\phi$. It can be shown that
$\kappa_{\phi\phi}=-\alpha'\beta'$,
$\kappa_{\phi\phi\phi}=-(2\alpha''\beta'+\alpha'\beta'')$, and
$\kappa_{\phi\phi\phi\phi}=-3\alpha''\beta''-3\alpha'''\beta'-\alpha'\beta'''$.
The gradient statistic takes the form
$S=n(\phi_{0}-\widehat{\phi})\alpha'(\phi_{0})(\beta(\phi_{0})+\bar d)$,
where $\bar{d}=n^{-1}\sum_{i=1}^{n}d(x_{i})$.
From $\eqref{c3}$, $\eqref{c2}$, and $\eqref{c1}$, we can write
\[
A_{1}= \frac{6}{\alpha'\beta'}\left\{2\left(\frac{\beta''}{\beta'}\right)^{2}
+\frac{\alpha''\beta''}{\alpha'\beta'}-\frac{\beta'''}{\beta'}\right\},
\]
\[
A_{2}=\frac{3}{\alpha'\beta'}\left[\frac{\beta''}{\beta'}\left(\frac{4\alpha''}{\alpha'}
-\frac{\beta''}{4\beta'}\right)+3\left\{\left(\frac{\alpha''}{\alpha'}\right)^2
+\left(\frac{\beta''}{\beta'}\right)^2\right\}-\left(\frac{\alpha'''}{\alpha'}-\frac{\beta'''}{\beta'}\right)\right],
\]
\[
A_{3}=\frac{5}{\alpha'\beta'}\left(\frac{\alpha''}{\alpha'}+\frac{\beta''}{2\beta'}\right)^{2}.
\]

We now present some special cases.
\begin{enumerate}
\item Normal ($\phi>0$, $\mu\in\mathbb{R}$, $x\in\mathbb{R}$):
\begin{itemize}
\item $\mu$ known: $\alpha(\phi)=1/(2\phi)$, $\xi(\phi)=\phi^{1/2}$,
$d(x)=(x-\mu)^2$, and $v(x)=-\log(2\pi)/2$. We have
$A_{1}=0$, $A_{2}=36$, and $A_{3}=40$. The first three moments of $S$
up to order $n^{-1}$ are $\mu_{1}'(S)=1$,
$\mu_{2}(S)=2(1+6/n)$, and $\mu_{3}(S)=8(1+29/n)$. The Bartlett-corrected
gradient statistic is $S^*=S\{1-(1-11S/3+2S^{2}/3)/(3n)\}$.
\item $\phi$ known: $\alpha(\mu)=-\mu/\phi$, $\xi(\mu)=\exp(\mu^2/2\phi)$,
$d(x)=x$, and $v(x)=-x^2/2 -\log(2\pi\phi)/2$. Here, $A_1=A_2=A_3=0$, as expected.
\end{itemize}
\item Inverse normal ($\phi>0$, $\mu>0$, $x>0$):
\begin{itemize}
\item $\mu$ known: $\alpha(\phi)=\phi$, $\xi(\phi)=1/\phi^{1/2}$,
$d(x)=(x-\mu)^2/(2\mu^2 x)$, and $v(x)=-\log(2\pi x^3)/2$. Here,
$A_1=24$, $A_2=30$, and $A_3=10$, and the three first moments of $S$ are
$\mu_{1}'(S)=1+2/n$, $\mu_{2}(S)=2+18/n$, and $\mu_{3}(S)=8+188/n$.
The Bartlett-corrected gradient statistic takes the form $S^*=S\{1-(S+2)(S+3)/(18n)\}$.
\item $\phi$ known: $\alpha(\mu)=\phi/(2\mu^2)$, $\xi(\phi)=\exp(-\phi/\mu)$,
$d(x)=x$, and $v(x)=-\phi/(2x^2)+\log(2\pi x^3)/2$. We have $A_1=0$ and
$A_{2}=A_{3}=45\mu/\phi$. The first three approximate moments of $S$ are $\mu_{1}'(S)=1$,
$\mu_{2}(S)=2+15\mu/(n\phi)$, and $\mu_{3}(S)=8+270\mu/(n\phi)$.
Also, $S^*=S\{1-\mu S(S-5)/(4  n \phi)\}$.
\end{itemize}
\item Gamma ($k$ known, $k>0$, $\phi>0$, $x>0$):
$\alpha(\phi)=\phi$, $\xi(\phi)=\phi^{-k}$, $d(x)=x$, and
$v(x)=(k-1)\log x-\log\Gamma(k)$, where $\Gamma(\cdot)$ denotes the gamma
function. We have $A_1=12/k$, $A_2=15/k$,  $A_3=5/k$, and first three approximate
moments $\mu_{1}'(S)=1+1/(nk)$, $\mu_{2}(S)=2+9/(nk)$, and $\mu_{3}(S)=8+94/(nk)$.
Also, $S^*=S\{1-(S+2)(S+3)/(36nk)\}$.
\item Truncated extreme value ($\phi>0$, $x>0$): $\alpha(\phi)=1/\phi$,
$\xi(\phi)=\phi$, $d(x)=\exp{(x)}-1$, and $v(x)=x$. We have
$A_1=0$, $A_2=12$, $A_3=20$,  $\mu_{1}'(S)=1$,
$\mu_{2}(S)=2+4/n$, $\mu_{3}(S)=8+88/n$, and $S^*=S\{1-(12-15S+2S^2)/(18n)\}$.
\item Pareto ($\phi>0$, $k>0$, $k$ known, $x>k$):
$\alpha(\phi)=1+\phi$, $\xi(\phi)=(\phi k^{\phi})^{-1}$,
and $v(x)=0$. Here, $A_1=12$, $A_2=15$, $A_3=5$,
$\mu_{1}'(S)=1+1/n$, $\mu_{2}(S)=2+9/n$, $\mu_{3}(S)=8+94/n$, and
$S^*=S\{1-(S+2)(S+3)/(36n)\}$.
\item Power ($\theta>0$, $\phi>0$, $\theta$ known, $x>\phi$):
$\alpha(\phi)=1-\phi$, $\xi(\phi)=\phi^{-1}\theta^{\phi}$, and
$v(x)=0$. The $A$'s, the first three approximate moments, and the Bartlett-type corrected
statistic coincide with those obtained for the Pareto distribution.
\item Laplace ($\theta>0$, $k\in\mathbb{R}$, $k$ known, $x\in\mathbb{R}$):
$\alpha(\theta) = \theta^{-1}$, $\zeta(\theta) = 2\theta$,
$d(x) = |x - k|$, and $v(x) = 0$. We have $A_1=0$, $A_2=18$, $A_3=20$,
$\mu_{1}'(S)=1$, $\mu_{2}(S)=2+6/n$, $\mu_{3}(S)= 8+112/n$,
and $S^{*}=S\{1-(3-11S+2S^{2})/(18n)\}$.
\end{enumerate}

\section{Models with two orthogonal parameters}\label{twofamily}

The two-parameter families of distributions under orthogonality of
the parameters \citep{CoxReid1987}, say $\phi$ and $\beta$, will be the subject of this section.
The null hypothesis under test is $\mathcal{H}_0:\phi=\phi_{0}$, where $\phi_{0}$ is a fixed value,
and $\beta$ acts as a nuisance parameter.
The orthogonality between $\phi$ and $\beta$ leads to considerable simplification
in the formulas of $A_1,$ $A_2$, and $A_3$.
Here, $\kappa_{\phi\phi\beta}=\Es(\partial^{3}\ell(\bm{\theta})/\partial\beta\partial\phi^2)$,
$\kappa_{\phi\phi\beta}^{(\beta)}=\partial \kappa_{\phi\phi\beta}/\partial\beta$, etc.
After some algebra, we have
\begin{equation}\label{ortg_coef}
A_{1}=A_{1\phi}+ A_{1\phi\beta},\qquad
A_{2}=A_{2\phi}+ A_{2\phi\beta},\qquad
A_{3}=-\frac{5\kappa_{\phi\phi\phi}^{2}}{4\kappa^{3}_{\phi\phi}},
\end{equation}
where $A_{1\phi}$ and $A_{2\phi}$ are equal to $A_1$ and $A_2$
given in \eqref{c3} and \eqref{c2}, respectively, and
\begin{align*}
A_{1\phi\beta}&=\frac{3\bigl\{4\kappa_{\phi\phi\beta}\kappa_{\phi\phi}^{(\beta)}
+\kappa_{\phi\beta\beta}\bigl(4\kappa_{\phi\phi}^{(\phi)}-\kappa_{\phi\phi\phi}\bigr)\bigr\}}
{\kappa_{\phi\phi}^{2}\kappa_{\beta\beta}}
+\frac{6\bigl(\kappa_{\phi\phi\beta\beta}-2\kappa_{\phi\phi\beta}^{(\beta)}
-2\kappa_{\phi\beta\beta}^{(\phi)}\bigr)}{\kappa_{\phi\phi}\kappa_{\beta\beta}}\\
&\quad+\frac{3\bigl\{2\kappa_{\phi\phi\beta}\bigl(2\kappa_{\beta\beta}^{(\beta)}
-\kappa_{\beta\beta\beta}\bigr)+\kappa_{\phi\beta\beta}\bigl(2\kappa_{\beta\beta}^{(\phi)}
-3\kappa_{\phi\beta\beta}\bigr)\bigr\}}{\kappa_{\phi\phi}\kappa_{\beta\beta}^2},
\end{align*}
\[
A_{2\phi\beta}=\frac{3\bigl(3\kappa_{\phi\phi\phi}\kappa_{\phi\beta\beta}
+\kappa^{2}_{\phi\phi\beta}\bigr)}{\kappa_{\phi\phi}^{2}\kappa_{\beta\beta}}.
\]
The expressions for $A_{1\phi\beta}$ and $A_{2\phi\beta}$ in \eqref{ortg_coef} 
can be regarded as the additional contribution
introduced in the expansion of the cumulative distribution function of the gradient statistic
owing to the fact that $\beta$ is unknown and has to be estimated from the data.
In the following, we present some examples.

\begin{ex}(Normal distribution)\end{ex}
Let $x_1,\ldots,x_n$ be a random sample from a normal distribution $N(\phi,\beta)$.
The gradient statistic can be written in the form
\[
S=n\frac{T_1/T_2}{1+T_1/T_2},
\]
where $T_1=n(\bar{x}-\phi_{0})^2$ and $T_2=\sum_{i=1}^{n}(x_i-\bar{x})^{2}$,
where $\bar{x}=n^{-1}\sum_{i=1}^{n}x_i$. Under the null hypothesis,
$T_1/\beta$ and $T_2/\beta$ are independent with distributions $\chi^2_{1}$ and $\chi^{2}_{n-1}$,
respectively. It can be shown that $n^{-1}S$ has a beta
distribution with parameters $1/2$ and $(n-1)/2$.
The first three exact moments of $S$ are $1$, $2(n-1)/(n+2)$, and $8(n-1)(n-2)/\{(n+2)(n+4)\}$,
respectively. Here, $A_1=A_3=0$ and $A_2=-18$. The first
three approximate moments of $S$ are
$\mu'(S)=1$, $\mu_{2}(S)=2-6/n$, and $\mu_{3}(S)=8-72/n$.
These moments differ from the approximate moments
only by terms of order less than $n^{-1}$.
The Bartlett-type corrected gradient statistic is $S^*=S\{1-(3-S)/(2n)\}$.

\begin{ex}(Bivariate two-parameter exponential distribution)\end{ex}
Let $x_{11},\ldots,x_{1n_{1}}$ and $x_{21},\ldots,x_{2n_{2}}$ be two independent random samples
from exponential distributions with means  $\mu$ and $\phi\mu$,
respectively. It can be shown that
$\phi$ and $\beta=\mu\phi^{1/2}$ are globally orthogonal.
The parameter of interest is $\phi$ -- the ratio of the means --
and the interest lies in testing $\mathcal{H}_{0}:\phi=1$, which
is equivalent to the equality of the two population means, against $\mathcal{H}_{a}:\phi\neq 1$.
We consider the balanced case ($n_1=n_2=n/2$, $n\geq 2$ even).
Let $\bar{x}_1$ and $\bar{x}_2$ be the sample means.
The log-likelihood function can be written as
\[
\ell(\phi,\beta)=-\log\beta-\frac{\bar{x}_{1}}{2\beta\phi^{-1/2}}-\frac{\bar{x}_2}{2\beta\phi^{^1/2}}.
\]
The gradient statistic for testing $\mathcal{H}_{0}$ takes the form
\[
S=\frac{n(\bar{x}_1 -\bar{x}_2)^2}{4\bar{x}_1\bar{x}},
\]
where $\bar{x}=(\bar{x}_1 +\bar{x}_2)/2$.
The cumulants of log-likelihood derivatives are
$\kappa_{\phi\phi}=-1/(4\phi^2)$, $\kappa_{\phi\phi\phi}=3/(4\phi^3)$,
$\kappa_{\phi\phi\phi\phi}=-45/(16\phi^{4})$,
$\kappa_{\beta\beta}=-1/\beta^2$, $\kappa_{\beta\beta\beta}=4/\beta^{3}$,
$\kappa_{\beta\phi\phi}=0$, $\kappa_{\beta\phi}=0$, $\kappa_{\phi\phi\beta}=1/(4\beta\phi^2)$,
and $\kappa_{\beta\beta\phi\phi}=-1/(2\beta^2 \phi^2)$.
From \eqref{ortg_coef}, we have $A_1=24$, $A_{2}=63$, and $A_{3}=45$.
The corrected gradient statistic becomes $S^{*}=S\{1-(S-1)(S-2)/(4n)\}$.

\begin{ex}(Two-parameter Birnbaum--Saunders distribution)\end{ex} 
The two-parameter Birnbaum--Saunders distribution was proposed by
\cite{BSa1969} and has cumulative distribution function in the form
$G(x)=\Phi(v)$, with $x > 0$, where $v=\phi^{-1}\rho(x/\beta)$, $\rho(z)= z^{1/2}-z^{-1/2}$, and
$\Phi(\cdot)$ is the standard normal cumulative distribution function;
$\phi>0$ and $\beta>0$ are the shape and scale parameters,
respectively. We wish to test $\mathcal{H}_{0}:\phi=\phi_{0}$ against the alternative
hypothesis $\mathcal{H}_{a}:\phi\neq\phi_{0}$,
where $\phi_{0}$ is a known positive constant. The
gradient statistic to test $\mathcal{H}_{0}$ is
\[
S=\frac{n(\widehat{\phi}-\phi_{0})}{\phi_{0}^3}\bigl\{\bar{s} + \bar{r} - (2 + \phi_{0}^2)\},
\]
where $\bar{s}=(n\widetilde{\beta})^{-1}\sum_{i=1}^nx_i$,
$\bar{r}=\widetilde{\beta}n^{-1}\sum_{i=1}^n x_i^{-1}$, and $\widetilde{\beta}$
is the maximum likelihood estimator of $\beta$ obtained under $\mathcal{H}_{0}$.
We have $\kappa_{\phi\phi} = -2/\phi^2$, $\kappa_{\phi\beta} = 0$,
and $\kappa_{\beta\beta} =  -\{1 + \phi(2\pi)^{-1/2}h(\phi)\}/(\phi^2\beta^2)$,
where $h(\phi) = \phi(\pi/2)^{1/2} - \pi\e^{2/\phi^2}\{1 - \Phi({2}/{\phi})\}$.
After some algebra, we obtain $A_{1\phi}=-3$, $A_{2\phi}=69/8$,
$A_{2\phi\beta}=-45(2+\phi^2)/[2\{1 + \phi(2\pi)^{-1/2}h(\phi)\}]$, $A_3=125/8$, and
\[
A_{1\phi\beta}=\frac{9-15\phi^2/2}{1 + \phi(2\pi)^{-1/2}h(\phi)}
-\frac{3(\phi^2+ 2)}{2\{1 + \phi(2\pi)^{-1/2}h(\phi)\}^2}
\biggl\{-4(1+\phi^2)+\frac{2(4+\phi^2)h(\phi)}{\phi\sqrt{2\pi}}\biggr\}.
\]
Since the necessary quantities to obtain the $A$'s were derived, a Bartlett-corrected
gradient statistic may be obtained from
Corollary \ref{col_corr_grad}. It is interesting to note
that the $A$'s do not depend on the unknown scalar parameter $\beta$.
Next, we shall present a small Monte Carlo simulation regarding the test of
the null hypothesis $\mathcal{H}_{0}:\phi=1$.

The simulations were performed by setting $\beta=1$ and sample sizes
ranging from 5 to 22 observations. All results are based on 10,000 replications.
The size distortions (i.e.~estimated minus nominal sizes) for the 5\%
nominal level of the gradient statistic and its
Bartlett-corrected version for different sample
sizes are plotted in Figure \ref{plots}(a). It is clear from this figure
that the Bartlett-corrected test displays smaller size distortions
than the original gradient test.

Finally, we set $n=10$ and consider
the first-order approximation ($\chi_{1}^2$ distribution)
for the distribution of the gradient statistic and the expansion
obtained in this paper. Figure \ref{plots}(b) presents the curves.
The difference between the curves is evident from this figure,
and hence, the $\chi_{1}^2$ distribution may not be a good
approximation for the null distribution of the gradient statistic in testing the
null hypothesis $\mathcal{H}_{0}:\phi=1$ for the two-parameter
Birnbaum--Saunders model if the sample is small.
\begin{figure}[!htp]
\centering
\includegraphics[scale=0.5]{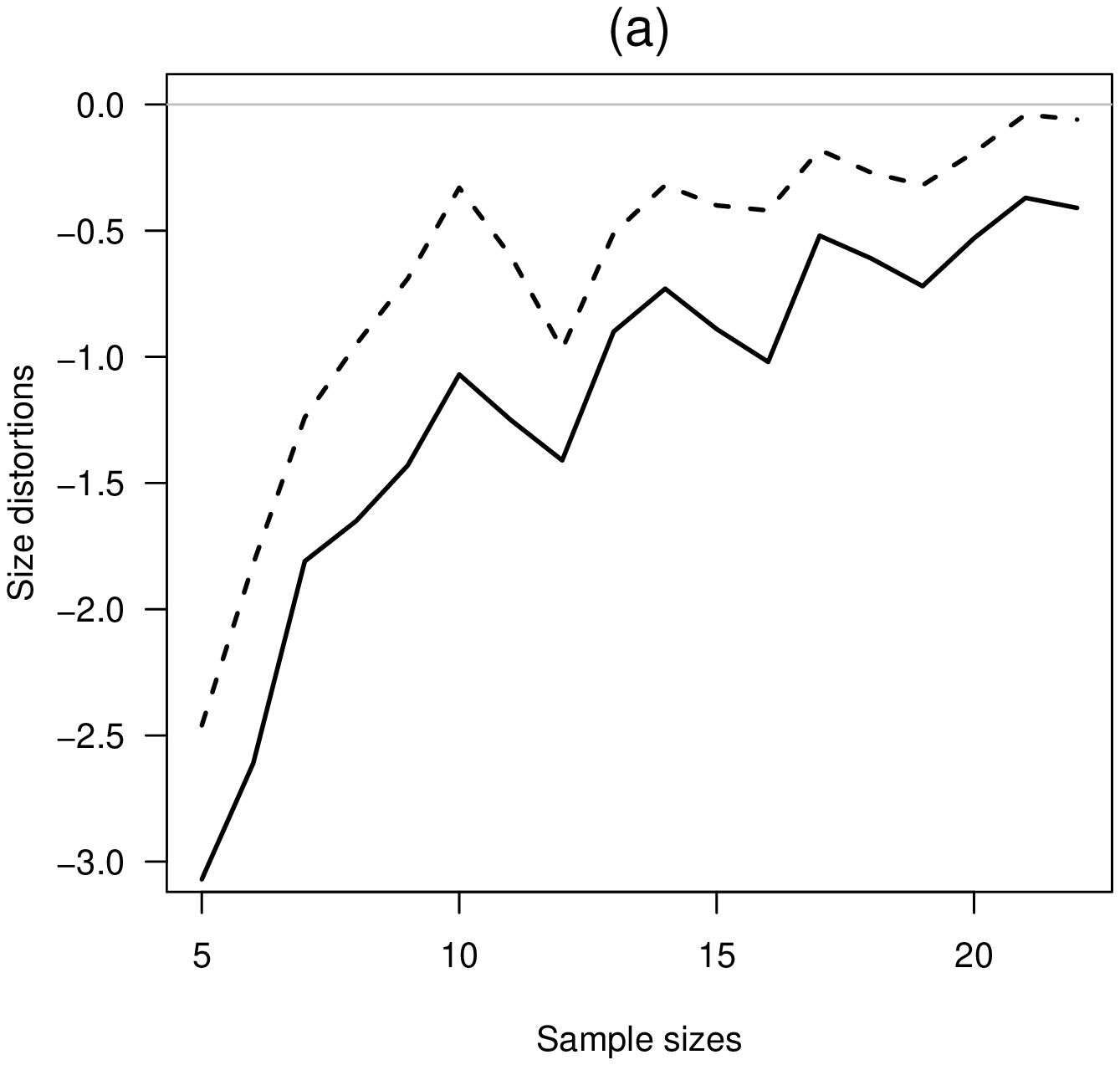}
\includegraphics[scale=0.5]{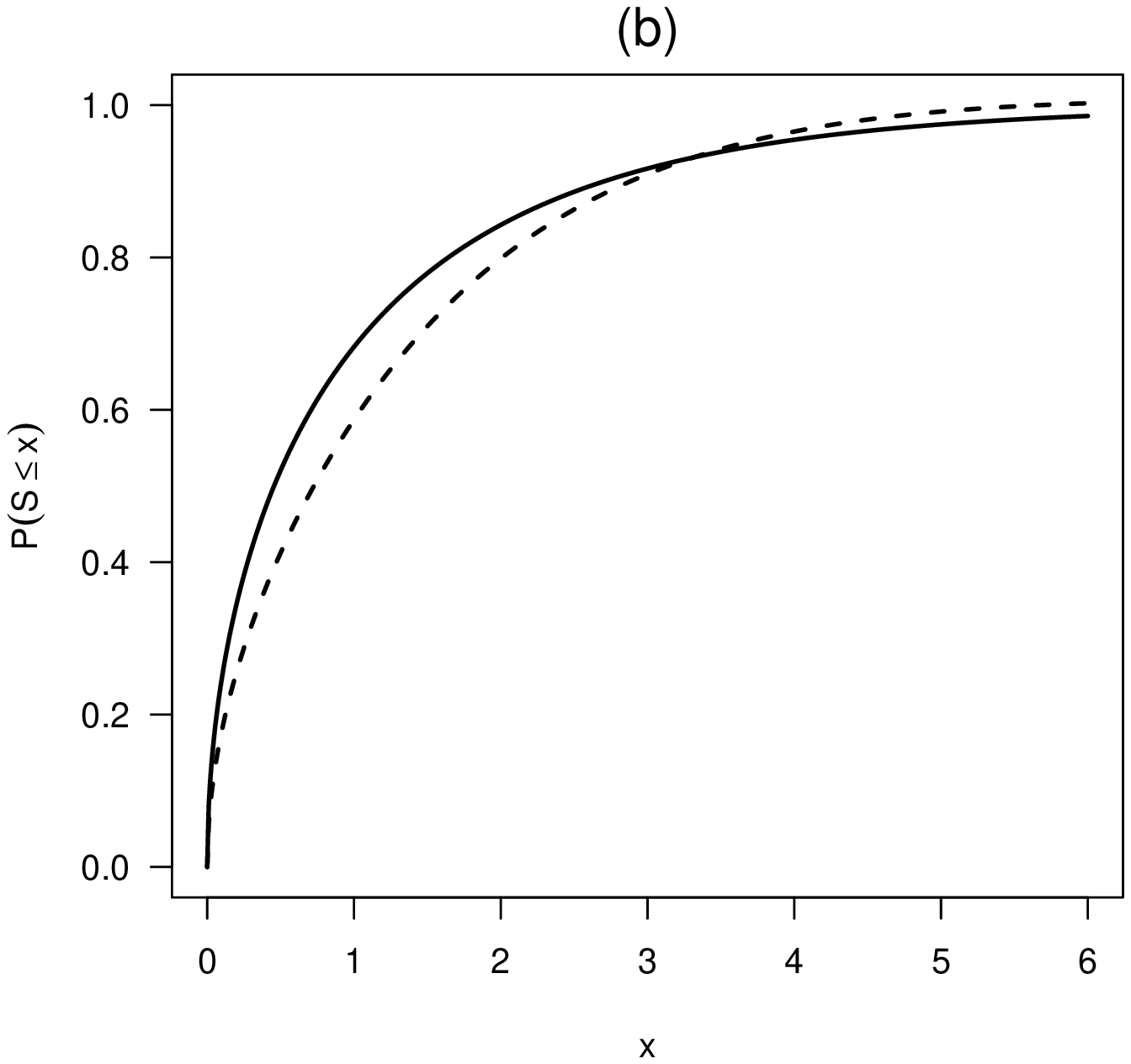}
\caption{(a) Size distortion of the gradient test (solid) and the
Bartlett-corrected gradient test (dashes);
(b) first-order approximation (solid) and expansion to order $n^{-1}$ (dashes)
of the null cumulative distribution function of the gradient statistic.}\label{plots}
\end{figure}

\section{Discussion}\label{conclusion}

\cite{LemonteFerrari2012a} showed that the gradient test can be
an interesting alternative to the classic large-sample tests,
namely the likelihood ratio, the Wald, and the Rao score tests,
since none is uniformly superior to the others in terms of second-order local power.
Additionally, as remarked before, the gradient statistic
does not require to obtain, estimate, or invert an information matrix,
unlike the Wald and the Rao score statistics. Its formal simplicity
is always an attraction.

The exact null distribution of the gradient statistic is
usually unknown and the test relies upon an asymptotic
approximation. The chi-square distribution is used as a large-sample approximation
to the true null distribution of this statistic. However, for small sample sizes,
the chi-square distribution may be a poor approximation to the true null
distribution; that is, the asymptotic approximation may deliver
inaccurate inference. In order to overcome this shortcoming,
an alternative strategy is to use a higher-order asymptotic theory.

The asymptotic expansion up to order $n^{-1}$ for the null distribution function of
the gradient statistic was derived in this paper. A Bayesian route based on the
shrinkage argument  \citep{GhoshMukerjee1991, MukerjeeReid2000} proved to be extremely useful
in this context. The expansion is very general in the sense that
the null hypothesis can be composite in the presence of nuisance parameters.
We show that the coefficients which define this expansion
depend on the joint cumulants of log likelihood derivatives for the full data.
Unfortunately, these coefficients are very difficult to interpret in generality.

\cite{CordeiroFerrari1991} showed that, quite generally, continuous statistics having
a chi-square distribution asymptotically can be modified by a suitable correction term
that makes the modified statistic have chi-square distribution to order $n^{-1}$. Their work
can be viewed as an extension of Bartlett corrections to the likelihood ratio statistic
\citep{Lawley1956} to other statistics having a chi-square distribution asymptotically. The
correction term comes from the coefficients of the $O(n^{-1})$ term in the expansion of
the cumulative distribution function of the test statistic in
such a way that it becomes better approximated by the
reference chi-square distribution. It is known as the Bartlett-type correction.
It is well known that Bartlett and Bartlett-type corrections have become a widely used
method for improving the large-sample chi-square
approximation to the null distribution of the likelihood ratio and Rao score statistics,
respectively. In recent years there has been a renewed interest in Bartlett factors and several
papers have been published giving expressions for computing these corrections for special models.
Some references are \cite{Zucker-et-al-2000}, \cite{LagosMorettin2004}, \cite{Tu-et-al-2005},
\cite{vanGiersbergen2009}, \cite{Bai-2009}, \cite{Lagos-et-al-2010}, and
\cite{Noma-2011}.

From the general expansion derived in this
paper and using results in \cite{CordeiroFerrari1991},
we also obtained a Bartlett-type correction factor for the gradient statistic. Our results are
very general and not tied to special classes of models. They allow the parameter
vector to be multidimensional and are valid regardless of whether nuisance parameters
are present or not. Additionally, as the coefficients in the expansion, and
consequently in the Bartlett-type
correction factor, are written as functions of cumulants of log-likelihood derivatives, they
can be obtained for all the classes of parametric models for which those cumulants
can be determined. Therefore, applications of our general results in several parametric models, such
as the generalised linear models and extensions,
can be studied in future research.

\section*{Acknowledgments}

We gratefully acknowledge grants from FAPESP and CNPq (Brazil).

{\small
\appendix

\section*{Appendix 1}\label{appendixA}
\subsection*{Proof of Theorem \ref{theorem1}}

Except when indicated, the indices $j$, $r$, $s$, $u$, $v$, and $w$
range over $1$ to $p$ and the indices $j'$, $r'$, $s'$, $u'$, $v'$, and $w'$
range over $1$ to $q$. Also, an array index repeated as both a superscript and a subscript
indicates an implied summation over the appropriate range.
Let $\lambda_{jr}=-\psi_{jr}=
-\{D_{j}D_{r}\ell(\bm{\bm{\theta}})\}_{\bm{\bm{\theta}}=\widehat{\bm{\bm{\theta}}}}$,
$\psi_{jrs}=\{D_jD_rD_s \ell(\bm{\bm{\theta}})\}_{\bm{\bm{\theta}}=\widehat{\bm{\bm{\theta}}}}$,
$\psi_{jrsu}=\{D_jD_rD_sD_u \ell(\bm{\bm{\theta}})\}_{\bm{\bm{\theta}}=\widehat{\bm{\bm{\theta}}}}$, etc.
The matrix $\bm{\Lambda}=((\lambda_{jr}))$ is the observed information matrix evaluated at $\widehat{\bm{\bm{\theta}}}$.
The partition of $\bm{\bm{\theta}}=(\bm{\bm{\theta}}_{1}^\top,\bm{\bm{\theta}}_{2}^\top)^\top$ induces the partition
\[
\bm{\Lambda}=((\lambda_{jr})) =
\begin{bmatrix}
\bm{\Lambda}_{11} & \bm{\Lambda}_{12} \\
\bm{\Lambda}_{21} & \bm{\Lambda}_{22}
\end{bmatrix},
\qquad
\bm{\Lambda}^{-1}=((\lambda^{jr})) =
\begin{bmatrix}
\bm{\Lambda}^{11} & \bm{\Lambda}^{12} \\
\bm{\Lambda}^{21} & \bm{\Lambda}^{22}
\end{bmatrix},
\]
where $\bm{\Lambda}^{-1}$ is the inverse of $\bm{\Lambda}$.
Let ${\bm{\Lambda}^{11}}^{-1}=((\lambda_{1w'j'}))$,
$\sigma^{jr}=\lambda^{jr}-\lambda^{jw'}\lambda_{1w'j'}\lambda^{j'r}$,
$\tau^{j j'}=\lambda^{jw'}\lambda_{1w'j'}$, $\sigma^{(1)}_{suvw}=\sigma^{su}\sigma^{vw}[3]$,
$\lambda^{(1)}_{j'r's'u'}=\lambda^{j'r'}\lambda^{s'u'}[3]$, and
$\lambda^{(2)}_{j'r's'u'v'w'}=\lambda^{j'r'}\lambda^{s'u'}\lambda^{v'w'}[15]$,
where $[\cdot]$ denotes a summation with the number in brackets indicating the number of terms
obtained by permutation of indices. For instance,
$\sigma^{su}\sigma^{vw}[3]=\sigma^{su}\sigma^{vw}+\sigma^{sv}\sigma^{uw}+\sigma^{sw}\sigma^{uv}$.
Let $\bm{\epsilon}=(\epsilon_{1},\ldots,\epsilon_q)^{\top}=n^{1/2}(\bm{\bm{\theta}}_{1}-\widehat{\bm{\bm{\theta}}}_{1})$,
$\Psi_{j'}^{(1)}=\psi_{jrs}\sigma^{rs}\tau^{j j'}/2$,
$\Psi_{j'r's'}^{(3)}=\psi_{jrs}\tau^{jj'}\tau^{rr'}\tau^{ss'}/6$,
\[
\Psi_{ j'r's'u'}^{(4)}=\frac{1}{24}\left\{\psi_{jrsu}+\sigma^{vw}(2\psi_{jrs}\psi_{uvw}
+3\psi_{jrv}\psi_{suw})\right\}\tau^{jj'}\tau^{rr'}\tau^{ss'}\tau^{uu'}.
\]

\begin{lemma}\label{lemma1}
An asymptotic expansion under the null hypothesis
for the gradient statistic $\eqref{est_grad}$ is
\begin{equation}\label{grad_expansion}
S=\bm{\epsilon}^{\top}{\bm{\Lambda}^{11}}^{-1}\bm{\epsilon}-\frac{3}{\sqrt{n}}\Psi_{j'r's'}^{(3)}\epsilon_{j'}\epsilon_{r'}\epsilon_{s'}
-\frac{4}{n}\Bigl(\Psi_{j'r's'u'}^{(4)}-\Psi_{j'r's'}^{(3)}\Psi_{u'}^{(1)}\Bigr)\epsilon_{j'}\epsilon_{r'}\epsilon_{s'}\epsilon_{u'}
+ o_p(n^{-1}).
\end{equation}
\end{lemma}
\begin{proof}
Using a procedure analogous to that of \cite{ChangMukerjee2011}, the result holds.
\end{proof}

Let $\pi=\pi(\bm{\bm{\theta}})$ be a prior density for  $\bm{\bm{\theta}}$,
$\pi_{j}=D_j\pi(\bm{\bm{\theta}})$, $\pi_{jr}=D_{j}D_{r}\pi(\bm{\bm{\theta}})$,
$\widehat{\pi}=\pi(\widehat{\bm{\bm{\theta}}})$, $\widehat{\pi}_j=\pi_j(\widehat{\bm{\bm{\theta}}})$,
$\widehat{\pi}_{jr}=\pi_{jr}(\widehat{\bm{\bm{\theta}}})$,
\[
\Psi_{ j'r'}^{(2)}=\left\{\frac{\widehat{\pi}_{jr}}{2\widehat{\pi}}+
\frac{1}{4}\psi_{jrsu}\sigma^{su}+\frac{1}{24}\bigl(2\psi_{jrs}\psi_{uvw}
+3\psi_{jsu}\psi_{rvw}\bigr)\sigma^{(1)}_{suvw}\right\}\tau^{jj'}\tau^{rr'},
\]
\[
\Gamma_{j'}^{(1)}=\Psi_{j'}^{(1)}+\frac{\widehat{\pi}_{j}}{\widehat{\pi}}\tau^{j j'},
\qquad
\Gamma_{j'r'}^{(2)}=\Psi_{j'r'}^{(2)}+\frac{1}{2\widehat{\pi}}\bigl(\psi_{jrs}\widehat{\pi}_u
+\psi_{jsu}\widehat{\pi}_r\bigr)\sigma^{su}\tau^{jj'}\tau^{rr'},
\]
\[
\Gamma_{j'r's'u'}^{(4)}=\Psi_{j'r's'u'}^{(4)}
+ \frac{\widehat{\pi}_{u}}{6\widehat{\pi}}\psi_{jrs}\tau^{jj'}\tau^{rr'}\tau^{ss'}\tau^{uu'}.
\]
From \cite{GhoshMukerjee1991}, \cite{ChangMukerjee2010} derive an expansion up to order  $n^{-1}$
for the marginal posterior density of $\bm{\epsilon}$, which takes the form
\begin{align}\label{pipost}  
\begin{split}
\pi_{post}(\bm{\epsilon})&=\phi_{q}(\bm{\epsilon};\bm{\Lambda}^{11})
\biggl[1+\frac{1}{\sqrt{n}}\bigl(\Gamma_{j'}^{(1)}\epsilon_{j'}+\Gamma_{j'r's'}^{(3)}\epsilon_{j'}\epsilon_{r'}\epsilon_{s'})\\
&\quad+\frac{1}{n}\Bigl\{\Gamma_{j'r'}^{(2)}\bigl(\epsilon_{j'}\epsilon_{r'}-\lambda^{j'r'}\bigr)
+\Gamma_{j'r's'u'}^{(4)}\bigl(\epsilon_{j'}\epsilon_{r'}\epsilon_{s'}\epsilon_{u'}-\lambda^{(1)}_{j'r's'u'}\bigr) \\
&\quad+\frac{1}{2}\Psi_{j'r's'}^{(3)}\Psi_{u'v'w'}^{(3)}\bigl(\epsilon_{j'}\epsilon_{r'}\epsilon_{s'}\epsilon_{u'}\epsilon_{v'}\epsilon_{w'}
-\lambda^{(2)}_{j'r's'u'v'w'}\bigr)\Bigr\}\biggr]+o(n^{-1}),
\end{split}
\end{align}
where $\phi_{q}(\bm{z};\bm{\Sigma})$ denotes the density of the
$q$-variate normal distribution with mean $\bm{0}$ and covariance matrix $\bm{\Sigma}$.

We now follow the Bayesian route described in \cite{MukerjeeReid2000}; see Appendix 2.\\
\noindent\textbf{Step 1.} The approximate posterior characteristic function of $S$ is
\[
M_{\pi}(t)=\Es_{\pi}\{\exp(\xi S)\}=\int\exp(\xi S)\pi_{post}(\bm{\epsilon}) \dd\bm{\epsilon},
\]
where $\xi = {\rm i}t$ with ${\rm i}=(-1)^{1/2}$.
From Lemma \ref{lemma1} and after some algebra, we can write
\begin{align*}
\exp(\xi S)\pi_{post}(\bm{\epsilon})&=(1-2\xi)^{-q/2}\phi_{q}\left(\bm{\epsilon};\frac{\bm{\Lambda}^{11}}{1-2\xi}\right)
\Biggl[1+\frac{1}{\sqrt{n}}\Bigl\{(1-3\xi)\Psi_{j'r's'}^{(3)}\epsilon_{j'}\epsilon_{r'}\epsilon_{s'}+\Gamma_{j'}^{(1)}\epsilon_j'\Bigr\}\\
&\quad+\frac{1}{n}\biggl[\frac{1}{2}\Psi_{j'r's'}^{(3)}\Psi_{u'v'w'}^{(3)}
\biggl\{\frac{1}{9}(1-3\xi)^{2}\epsilon_{j'}\epsilon_{r'}\epsilon_{s'}\epsilon_{u'}\epsilon_{v'}\epsilon_{w'}-\lambda^{(2)}_{j'r's'u'v'w'}\biggr\}\\
&\quad-\Bigl[\xi\bigl\{4\Psi_{j'r's'u'}^{(4)}+\Psi_{j'r's'}^{(3)}\bigl(3\Gamma_{u'}^{(1)}-4\Psi_{u'}^{(1)}\bigr)\bigr\}
-\Gamma_{j'r's'u'}^{(4)}\Bigr]\epsilon_{j'}\epsilon_{r'}\epsilon_{s'}\epsilon_{u'}\\
&\quad+\Gamma_{j'r'}^{(2)}\bigl(\epsilon_{j'}\epsilon_{r'}
-\lambda^{j'r'}\bigr)-\Gamma_{j'r's'u'}^{(4)}\lambda^{(1)}_{j'r's'u'}\biggr]\Biggr]+o_p(n^{-1}).
\end{align*}
Now, by writing $\xi=-{\frac{1}{2}(1-2\xi)+\frac{1}{2}}$,
$\xi^2={\frac{1}{4}}(1-2\xi)^2-{\frac{1}{2}}(1-2\xi)+{\frac{1}{4}}$,
and assuming that $\bm{\bm{\theta}}$ is in the interior of the support of $\pi$,
we obtain after some algebra
\begin{equation}\label{fgm1}
M_{\pi}(t)=(1-2\xi)^{-q/2}\left\{1+\frac{1}{n}\sum_{i=0}^{3}H_{i}(1-2\xi)^{-i}\right\}+o_p(n^{-1}),
\end{equation}
where $H_{0}=-(H_{1}+H_{2}+H_{3})$,
\begin{align*}
H_{1}&=\frac{9}{8}\Psi_{j'r's'}^{(3)}\Psi_{u'v'w'}^{(3)}\lambda^{(2)}_{j'r's'u'v'w'}+\Gamma_{ j'r'}^{(2)}\lambda^{j'r'} \\
&\quad+\lambda^{(1)}_{j'r's'u'}\left\{2\bigl(\Psi_{j'r's'u'}^{(4)}-\Psi_{ j'r's'}^{(3)}\Psi_{u'}^{(1)}\bigr)
+\frac{3}{2}\Psi_{j'r's'}^{(3)}\Gamma_{u'}^{(1)}\right\},
\end{align*}
\begin{align*}
H_{2}&=-\frac{3}{4}\Psi_{j'r's'}^{(3)}\Psi_{u'v'w'}^{(3)}\lambda^{(2)}_{j'r's'u'v'w'} \\
&\quad+\lambda^{(1)}_{j'r's'u'}\left\{\Gamma_{j'r's'u'}^{(4)}-2\bigl(\Psi_{j'r's'u'}^{(4)}-\Psi_{j'r's'}^{(3)}\Psi_{u'}^{(1)}\bigr)
-\frac{3}{2}\Psi_{j'r's'}^{(3)}\Gamma_{u'}^{(1)}\right\},
\end{align*}
\[
H_{3}=\frac{1}{8}\Psi_{j'r's'}^{(3)}\Psi_{u'v'w'}^{(3)}\lambda^{(2)}_{j'r's'u'v'w'}.
\]

\noindent\textbf{Step 2.} Let $\bar{\pi}(\cdot)$ be an auxiliary prior density for $\bm{\theta}$
satisfying the conditions in \cite{BickelGhosh1990}.
We now obtain an approximate posterior characteristic function  of $S$ under the prior $\bar{\pi}(\cdot)$,
say $M_{\bar{\pi}}(t)$. From \eqref{fgm1}, we have
\begin{equation*}\label{4}
M_{\bar{\pi}}(t)=(1-2\xi)^{-q/2}\left\{1+\frac{1}{n}\sum_{i=0}^{3}\bar{H}_{i}(1-2\xi)^{-i}\right\}+o_p(n^{-1}),
\end{equation*}
where $\bar{H}_{i}$ denotes the counterpart of $H_{i}$ obtained by
replacing  $\pi(\cdot)$ with $\bar{\pi}(\cdot)$.
After some algebra, we have
\[
\Delta(\bm{\bm{\theta}})=\Es_{\bm{\theta}}(M_{\bar{\pi}})=(1-2\xi)^{-q/2}
\left\{1+\frac{1}{n}\sum_{i=0}^{3}\bar{J}_{i}(1-2\xi)^{-i}\right\}+o(n^{-1}),
\]
where $\bar{J}_{0}=-(\bar{J}_{1}+\bar{J}_{2}+\bar{J}_{3})$,
\begin{align*}
\bar{J}_{1}&=\frac{1}{32}\kappa_{jrs}\kappa_{uvw}\bigl(9m^{jr}m^{su}m^{vw}+6m^{ju}m^{rv}m^{sw}\bigr)
+\frac{1}{4}\bigl(\kappa_{jrsu}+3\kappa_{jrv}\kappa_{suw}a^{vw}\bigr)m^{jr}m^{su}\\
&\quad+\frac{3}{8}\kappa_{jrs}\kappa_{uvw}a^{vw}m^{jr}m^{su}+\frac{3}{4}\kappa_{jrs}m^{jr}m^{su}\frac{\bar{\pi}_{u}}{\bar{\pi}}\\
&\quad+\frac{1}{2}m^{jr}\biggl\{\frac{\bar{\pi}_{jr}}{\bar{\pi}}+\frac{1}{2}\kappa_{jrsu}a^{su}+\frac{1}{4}(2\kappa_{jrs}\kappa_{uvw}
+3\kappa_{jsu}\kappa_{rvw}a^{su}a^{vw})  \\
&\qquad+\frac{\bar{\pi}_{u}}{\bar{\pi}}\kappa_{jrs}a^{su}  +\frac{\bar{\pi}_{r}}{\bar{\pi}}\kappa_{jsu}a^{su}\biggr\},
\end{align*}
\begin{align*}
\bar{J}_{2}&=-\frac{1}{48}\kappa_{jrs}\kappa_{uvw}\bigl(9m^{jr}m^{su}m^{vw}+6m^{ju}m^{rv}m^{sw}\bigr)
-\frac{1}{4}\bigl(\kappa_{jrsu}+3\kappa_{jrv}\kappa_{suw}a^{vw}\bigr)m^{jr}m^{su}\\
&\quad-\frac{3}{8}\kappa_{jrs}\kappa_{uvw}a^{vw}m^{jr}m^{su}-\frac{3}{4}\kappa_{jrs}m^{jr}m^{su}\frac{\bar{\pi}_{u}}{\bar{\pi}}\\
&\quad+3m^{jr}m^{su}\left[\frac{1}{24}\{\kappa_{jrsu}+\bigl(2\kappa_{jrs}\kappa_{uvw}
+3\kappa_{jrv}\kappa_{suw}\bigr)a^{vw}\}+\frac{1}{6}\kappa_{jrs}\frac{\bar{\pi}_u}{\bar{\pi}}\right],
\end{align*}
\[
\bar{J}_{3}=\frac{1}{288}\kappa_{jrs}\kappa_{uvw}\bigl(9m^{jr}m^{su}m^{vw}+6m^{ju}m^{rv}m^{sw}\bigr).
\]

\noindent\textbf{Step 3.} We now compute
\[
\int\Delta(\bm{\theta})\bar{\pi}(\bm{\theta}) \dd\bm{\theta}
=(1-2\xi)^{-q/2}\left\{1+\frac{1}{n}\sum_{i=0}^{3}(1-2\xi)^{-i}
\int\bar{J}_{i}\bar{\pi}(\bm{\theta}) \dd\bm{\theta}\right\}+o(n^{-1}),
\]
by integrating the $\bar{J}$'s with respect to $\bar{\pi}$.
After integrating each term that depends on the prior distributions
and by allowing $\bar{\pi}(\cdot)$ to converge weakly to the degenerate
prior at the true value of $\bm{\theta}$, we arrive at
\[
\Es_{\bm{\theta}}\{\exp(\xi S)\}=(1-2\xi)^{-q/2}\left\{1+n^{-1}\sum_{i=0}^{3}\bar{A}_{i}(1-2\xi)^{-i}\right\}+o(n^{-1}),
\]
where the $\bar{A}$'s are functions of cumulants of log-likelihood derivatives.
By writing  $d=2\xi/(1-2\xi)$ and  using the fact that
$\sum_{i=0}^{3}\bar{A}_{i}=0$, we arrive at
\begin{equation}\label{cf_approx}
M(t)=(1-2\xi)^{-q/2}\left\{1+\frac{1}{24n}(A_{1}d+A_{2}d^2+A_{3}d^{3})\right\}+o(n^{-1}),
\end{equation}
with $A_1=24(\bar{A}_{1}+2\bar{A}_{2}+3\bar{A}_3)$, $A_2=24(\bar{A}_{2}+3\bar{A}_3)$, and
$A_3=24\bar{A}_3$. We can write
\begin{align*}
A_{1}&=12D_j D_{r}m^{jr}-6D_{u}(\kappa_{jrs}m^{jr}m^{su})
-12D_{u}\bigl(\kappa_{jrs}m^{jr}a^{su}\bigr)-12D_{r}\bigl(\kappa_{jsu}m^{jr}a^{su}\bigr)\\
&\quad+6\kappa_{jrsu}m^{jr}a^{su}+3\kappa_{jrs}\kappa_{uvw}\bigl(m^{jr}m^{su}a^{vw}+2m^{jr}a^{su}a^{vw}\bigr)
+9\kappa_{jsu}\kappa_{rvw}m^{jr}a^{su}a^{vw},
\end{align*}
\begin{align*}
A_{2}&=6D_{u}\bigl(\kappa_{jrs}m^{jr}m^{su}\bigr)-3\kappa_{jrs}\kappa_{uvw}\biggl(m^{jr}m^{su}a^{vw}
+\frac{3}{4}m^{jr}m^{su}m^{vw}+\frac{1}{2}m^{ju}m^{rv}m^{sw}\biggr)\\
&\quad-3\kappa_{jrsu}m^{jr}m^{su}-9\kappa_{jrv}\kappa_{suw}m^{jr}m^{su}a^{vw},
\end{align*}
\[
A_{3}=\frac{1}{12}\kappa_{jrs}\kappa_{uvw}\bigl(9m^{jr}m^{su}m^{vw}+6m^{ju}m^{rv}m^{sw}\bigr).
\]
Inverting $M(t)$ in \eqref{cf_approx} and
interchanging the indices in a suitable manner,  after some algebra,
we arrive at the expression for $A_1$, $A_2$, and $A_3$
as given in Theorem  $\ref{theorem1}$.

\section*{Appendix 2}\label{appendixB}
\subsection*{The Shrinkage Argument}

Let $\bm{x}=(x_1,\ldots,x_n)^{\top}$ be a random vector with density $f(\cdot,\bm{\theta})$,
where $\bm{\theta}\in\bm{\Theta}$ is a  $p$-dimensional parameter
and $\bm{\Theta}\subseteq\mathbb{R}^{p}$ is an open subset of the
Euclidean space. Let $Q(\cdot,\bm{\theta})$ be a measurable function. Assume that $Q$ is
continuous for all $\bm{\theta}$ and that its expectation exists. A Bayesian route
for obtaining $\Es_{\bm{\theta}}\{Q(\cdot,\bm{\theta})\}$ based on a shrinkage argument
involves the three steps described below.
\begin{description}
\item[Step 1.] Obtain  $\Es_{\pi}\{Q(\bm{\theta},\bm{X})|\bm{X}= \bm{x}\}$, the posterior
expectation of $Q$ under the prior $\pi(\cdot)$ for $\bm{\theta}$.
\item[Step 2.] Find $\Es_{\bm{\theta}}[\Es_{\pi}\{Q(\bm{\theta},\bm X)|\bm X= \bm x\}]=\Delta(\bm{\theta})$,
for $\bm{\theta}\in int_s(\pi)$, where $int_s(\pi)$ denotes the interior of the support of $\pi$.
\item[Step 3.] Integrate $\Delta(\bm{\theta})$ with respect to $\pi(\cdot)$
and allow $\pi(\cdot)$ to converge weakly to the degenerate prior
at $\bm{\theta}$, where $\bm{\theta}\in int_s(\pi)$. This yields
$\Es_{\bm{\theta}}\{Q(\bm{X},\bm{\theta})\}$. 
\end{description}
A detailed justification can be found in \cite{MukerjeeReid2000}.
} 

\end{document}